\documentclass[english]{article}
\usepackage[T1]{fontenc}
\usepackage[latin9]{inputenc}
\usepackage{float}
\usepackage{amsmath}
\usepackage{amsthm}
\usepackage{amssymb}
\usepackage{graphicx,xcolor}

\makeatletter
\theoremstyle{plain}
\newtheorem{thm}{\protect\theoremname}
  \theoremstyle{plain}
  \newtheorem{cor}[thm]{\protect\corollaryname}
  \theoremstyle{plain}
  \newtheorem{lem}[thm]{\protect\lemmaname}
  \theoremstyle{plain}
  \newtheorem{prop}[thm]{\protect\propositionname}
  
 \theoremstyle{remark} 
 \newtheorem{rem}{Remark}[section]

\makeatother

\usepackage{babel}
  \providecommand{\corollaryname}{Corollary}
  \providecommand{\lemmaname}{Lemma}
  \providecommand{\propositionname}{Proposition}
\providecommand{\theoremname}{Theorem}

\newcommand\vv{\textsc{v}}
\newcommand{\G}{{\mathcal{G}}}
\newcommand{\R}{{\mathbb{R}}}
\begin{document}

\title{Peaked and low action solutions of NLS equations
on graphs with terminal edges}
\author{S. Dovetta, M. Ghimenti, A. M. Micheletti, A. Pistoia}

\maketitle

\begin{abstract}
	We consider the nonlinear Schr\"odinger equation with focusing power--type nonlinearity on compact graphs with at least one terminal edge, i.e. an edge ending with a vertex of degree 1. On the one hand, we introduce the associated action functional and we provide a profile description of positive low action solutions at large frequencies, showing that they concentrate on one terminal edge, where they coincide with suitable rescaling of the unique solution to the corresponding problem on the real line. On the other hand, a Ljapunov--Schmidt reduction procedure is performed to construct one--peaked and multipeaked positive solutions with sufficiently large frequency, exploiting the presence of one or more terminal edges.
\end{abstract}

\section{Introduction}

Metric graphs (or networks) are locally one--dimensional structures built of several intervals, the {\em edges}, glued together at some of their endpoints, the {\em vertices}. The specific way in which the edges are joined determines the topology of the graph. When a differential operator acting on functions supported on the graph is defined, we also speak of {\em quantum graphs}. 

The birth of quantum graphs can be traced back to the first half of the Fifties of the last century \cite{RS}, when the spectral analysis of Schr\"odinger operators on a network modelling molecular bonds has been proposed to investigate the behaviour of valence electrons in a naphthalene molecule. Since then, graphs have been assumed to provide a meaningful tool to model the dynamics of systems confined to ramified domains.

Despite the fact that, in general, to rigorously justify the graph approximation is still an open problem (see for instance \cite{grieser,molchanov} as well as \cite{CF,EP} and references therein), the last decades have been witnessing a renewed interest in the theory of quantum graphs, mainly driven by a wide variety of applications, e.g. Josephson junctions, propagations of signals, nonlinear optics and so on. Among these, the most prominent topic is probably given by the theory of Bose--Einstein condensates, that contributes to gather the focus on {\em nonlinear Schr\"odinger} (NLS) equations as
\begin{equation}
\label{tnlse}
-i\partial_t\psi(x,t)=\Delta_x \psi(x,t)+|\psi(x,t)|^{p-1}\psi(x,t)\,.
\end{equation}
Particularly, many efforts have been profuse in the analysis of {\em standing waves} of \eqref{tnlse}, i.e. solutions of the form $\psi(x,t)=e^{i\lambda t}u(x)$, for suitable $\lambda\in\R$ and $u$ solving the associated stationary equation
\begin{equation}
\label{nlseq}
-u''+\lambda u=|u|^{p-1}u\,.
\end{equation}
First investigations have been developed on specific examples of graphs with half--lines, such as star graphs (see for instance \cite{ACFN-jde14,ACFN-aihp14,N}) and the tadpole graph, \cite{NPS}. Later, the problem has been addressed on general non--compact graphs with half--lines, for which a quite well--established theory of existence of standing waves is nowadays available (see the series of works \cite{AST-CVPDE,AST-JFA,AST-CMP} for the case of the nonlinearity extended to the whole graph, and \cite{DT-p,DT,ST-JDE,ST-NA,T-JMAA} for the counterpart with nonlinearities restricted to the compact core). Broadening the discussion, several results have been accomplished also on compact graphs \cite{CDS,D-jde,MP-AMRX} and periodic graphs \cite{AD,ADST,D-nodea,Pa,PS}. Furthermore, similar investigations have been recently initiated on different families of nonlinear equations too, i.e. nonlinear KdV equation, \cite{MNS}, and nonlinear Dirac equation \cite{BCT,BCT1}.

From the standpoint of Critical Point Theory, solutions of \eqref{nlseq} can be identified at least in two different ways. On the one hand, one can search for critical points of the {\em energy functional} $E:H^1(\G)\to\R$
\[
E(u,\G):=\frac{1}{2}\int_\G|u'|^2\,dx-\frac{1}{p+1}\int_\G|u|^{p+1}\,dx
\]
in the constrained space of functions $u\in H^1(\G)$ with {\em prescribed mass }$\rho^2$, that is
\[
\int_\G|u|^2\,dx=\rho^2\,.
\]
This is for instance the general framework of \cite{AST-CVPDE,AST-JFA,AST-CMP} and related works, where it has been shown that the problem is sensitive both to topological and metric properties of the graph.

On the other hand, given $\lambda>0$, one can look for unconstrained critical points of the {\em action functional} $I:H^1(\G)\to\R$
\begin{equation}
\label{action}
I(u,\G):=\frac{1}{2}\int_\G|u'|^2\,dx-\frac{1}{p+1}\int_\G|u|^{p+1}\,dx+\frac{\lambda}{2}\int_\G|u|^2\,dx\,.
\end{equation}
This approach has been exploited in \cite{Pa} in the case of periodic graphs, and in \cite{GKP,KP-JDE,KP-JPA} on star--graphs. Precisely, in \cite{Pa}, minimization on a generalized Nehari manifold is performed to show existence of least action solutions, whereas in \cite{GKP,KP-JDE,KP-JPA} the focus is set on stability properties of specific critical points of the functional. 

Our work here fits in the investigation of the action functional \eqref{action}. Let us now describe informally the main results of the paper, redirecting to the next section for the precise setting and statements.

In what follows, we restrict our attention to compact graphs with at least one terminal edge, that is an edge ending with a vertex of degree 1. Our aim is twofold.

On one side, it is easy to show that a solution of \eqref{nlseq} can always be found minimizing the action on a suitable Nehari manifold. Hence, we concentrate on low action positive solutions and, given $\lambda$ large enough, we provide a profile description for such states. Specifically, we show that these solutions are strongly affected by the presence of a terminal edge, as it can be proved that they reach their maximum at a vertex of degree 1, whereas they are small in some norm outside of the corresponding terminal edge (Theorem \ref{thm:profile}).

On the other side, as soon as $\lambda$ is sufficiently large, a Ljapunov--Schmidt reduction procedure is performed. Exploiting the topological assumption ensuring at least a terminal edge, we construct one--peaked (Theorem \ref{thm:1picco}) and multipeaked (Theorem \ref{thm:multi}) solutions to \eqref{nlseq}, i.e. solutions with one or more maximum points at the vertices of degree 1, respectively, and negligibly small on the rest of the graph. 

The existence of such highly concentrated states testifies the dependence of the problem on the topology of the underlying graphs, which is a common feature of these kind of problems on graphs (just to name an example in the framework of compact graphs, the role of terminal edges in existence issues for the mass--constrained case has been pointed out in \cite{D-jde}). 

Let us highlight that several perspectives can be raised following up the aforementioned results. It is for instance unclear if functions
sharing the minimal action can be further characterized, and if the metric of $\G$ affects such
minimizers. In the case of multiple terminal edges, we expect solutions of least action to attain their peak on the longest among these edges, but up to now we are not able to provide a proof of this conjecture.

Another natural question concerns the possibility of adapting our construction to graphs without terminal edges, exhibiting states with peaks in the interior of any given edge. With respect to this, we believe that the profile description of low action solutions as given in Theorem \ref{thm:profile} generalizes straightforwardly, leading to a similar result for graphs with no terminal edges and solutions concentrated on an internal edge. Conversely, it seems to us that further work might be necessary for the Ljapunov--Schmidt scheme of Theorems \ref{thm:1picco}--\ref{thm:multi}. Indeed, it is not clear what suitable model function has to be considered in the absence of terminal edges, so that we expect nontrivial modifications of the argument to be required so to build peaked solutions with maximum points inside a general edge.

Finally, it remains an open problem to understand whether a profile description analogous to the one in Theorem \ref{thm:profile} can be 
given when we minimize the energy functional under a mass constraint.We notice that, in the context of mass--constrained critical points, solutions attaining their maximum only inside a given edge have been constructed in \cite{AST-bound}, provided the mass is sufficiently large. In that paper, such existence result is achieved through the analysis of a doubly--constrained minimization problem. We wonder whether the methods we introduce in the present work could be adapted to recover and generalize those conclusions.

\medskip
The paper is organised as follows. In Section \ref{sec:pre} we recall some known facts and we state in details our main results. Section \ref{sec:Profile} is devoted to the profile description of low action solutions, developing the proof of Theorem \ref{thm:profile}, whereas Section \ref{sec:LS} carries on the construction of peaked solutions as in Theorems \ref{thm:1picco}--\ref{thm:multi}.

\section{Setting and main results}
\label{sec:pre}

Before going further, let us briefly recall some definitions and some notation about metric graphs (for a standard reference see for instance the monograph \cite{BK}).

Throughout the paper, $\mathcal{G}=(V,\,E)$ denotes a compact graph, i.e. the union of a finite number of vertices $\vv\in V$ and edges $e\in E$, each one identified with an interval $I_{e}=[0,l_{e}]$ of finite length $l_{e}>0$. The degree of a vertex is the number of
edges entering it.

In what follows, we always assume that $\G$ has at least one terminal edge, that is an edge ending with a vertex of degree 1.

A function $u$ supported on $\mathcal{G}$ can be viewed as a bunch
of functions $u=(u_{e})_{e\in E}$ where 
\[
u_{e}:I_{e}\rightarrow\mathbb{R}
\]
and, since the graph $\mathcal{G}$ inherits the metric from its
edges, we can easily define the Lebesgue space $L^p(\G)$ of functions such that
\[
|u|_{L^{p}(\mathcal{G})}^p:=\sum_{e}|u_{e}|_{L^{p}([0,l_{e}])}^p<+\infty\,.
\]
A continuous function on $\mathcal{G}$ is a function which is continuous
on every edge and such that, if two edges $e$ and $f$ meet at
a vertex $\vv\in V$, then $u_{e}$ and $u_{f}$ have the same value
at $\vv$. We can also define the Sobolev space $H^{1}(\mathcal{G})$
as follows: 
\[
H^{1}(\mathcal{G})=\left\{ u=(u_{e})_{e}\in C^{0}(\mathcal{G})\ ,\ u_{e}\in H^{1}(I_{e})\text{ for all }e\right\} ,
\]
endowed with the norm
\[
\|u\|_{H^{1}(\mathcal{G})}^{2}=\sum_{e}\|u_{e}\|_{H^{1}(I_{e})}^{2}.
\]
Finally, the following Kirchhoff
condition is considered at the vertices of $\G$
\[
\sum_{e\prec \vv}\frac{du_{e}}{dx}(\vv)=0\,,\qquad\forall\vv\in V\,.
\]
Here, the symbol $e\prec \vv$ indicates every edge $e$ incident at $\vv$, and
we use the convention that 
\[
\frac{du_{e}}{dx}(\vv)=u'(0)\text{ or }\frac{du_{e}}{dx}(\vv)=-u'(l_{e})
\]
according to whether the $x$ coordinate is equal to $0$ or $l_{e}$
at $v$. 

We want to study the positive
solutions of the problem 
\begin{equation}
\left\{ \begin{array}{cc}
-u''+\lambda u=|u|^{p-1}u & \text{ in }\mathcal{G}\\
\sum_{e\prec \vv}\frac{du_{e}}{dx}(\vv)=0 & \forall \vv\in V\\
|u|_{L^{2}(\mathcal{G})}=\rho & u\in H^{1}(\mathcal{G})
\end{array}\right.\label{eq:P}
\end{equation}
on a graph with terminal edges. Here $p>1$ and $\lambda,\rho>0$
are given. Since $\lambda>0$, we endow $H^{1}(\mathcal{G})$ with
the following equivalent scalar product
\[
\left\langle u,v\right\rangle _{\lambda}=\int_{\mathcal{G}}u'(x)v'(x)dx+\lambda\int_{\mathcal{G}}u(x)v(x)dx.
\]
From now on, unless otherwise specified, we will always consider this
product (and its related norm $\|\cdot\|_{\lambda}$) as the scalar
product (and the norm) on $H^{1}(\mathcal{G})$

In order to find positive solutions of (\ref{eq:P}), we modify the action functional considering 
$I^+:H^1(\G)\to\R$
\begin{equation}
\label{action+}
I^+(u,\G):=\frac{1}{2}\int_\G|u'|^2\,dx-\frac{1}{p+1}\int_\G|u^+|^{p+1}\,dx+\frac{\lambda}{2}\int_\G|u|^2\,dx\,.
\end{equation}
In fact, any critical point of $I^+$ is a solution of 
\begin{equation}
\left\{ \begin{array}{cc}
-u''+\lambda u=(u^+)^{p} & \text{ in }\mathcal{G}\\
\sum_{e\prec \vv}\frac{du_{e}}{dx}(\vv)=0 & \forall \vv\in V\\
|u|_{L^{2}(\mathcal{G})}=\rho & u\in H^{1}(\mathcal{G})
\end{array}\right.\label{eq:P+}
\end{equation}
and, of course any positive solution of (\ref{eq:P}) is also a solution of (\ref{eq:P+}). 
One can prove also that any nontrivial solution of (\ref{eq:P+}) is a positive solution of 
(\ref{eq:P}), so any nontrivial critical point of $I^+$ is a positive solution of (\ref{eq:P}). 
To do that, it is sufficient to show that any solution 
$\bar u\not\equiv0$ of 
(\ref{eq:P+}) is strictly positive. We know that $\bar u$ as a minimum point $P\in \mathcal{G}$. 
By contradiction, let us suppose that $\bar u(P)\le 0$. If $P$ lies in the interior of some edge, then
$$
0\le \bar{u}''(P)=\lambda \bar{u}-(\bar{u}^+)^{p}=\lambda \bar{u}\le 0, 
$$ 
so $\bar{u}''(P)=\bar{u}'(P)=\bar{u}(P)=0$ and by the Cauchy theorem $\bar{u}\equiv0$ on the whole edge. Then, by Kirchhoff node condition we can prove that $\bar{u}\equiv 0 $ on $\mathcal{G}$, 
which is a contradiction. 
On the other hand, if $P$ coincides with a terminal vertex, we have that either $\bar{u}\equiv 0$ or $\bar{u}'(P)=0$, $\bar{u}(P)<0$ and $\bar{u}''(P)>0$, and $P$ cannot be a minimum point.  
If $P$ coincides with and internal vertex, a similar argument applies and we get the proof. 

Now, for every $\lambda>0$, let 
\begin{equation}\label{Jlambda}
J_\lambda(u):=\lambda^{\frac{1}{2}-\frac{p+1}{p-1}}I^+(u)
\end{equation}
be the renormalized action functional,  and consider the associated Nehari manifold
\[
\mathcal{N}_{\lambda}:=\left\{ u\in H^{1}(\mathcal{G})\smallsetminus\left\{ 0\right\} \ :\ J_{\lambda}'(u)[u]=0\right\}\,.
\] 
It is standard to prove that  $\mathcal{N}_{\lambda}$ is a natural constraint, i.e. that any nontrivial solution of (\ref{eq:P}) is a critical point of $J_\lambda$ on $\mathcal{N}_{\lambda}$.

Finally, we recall some useful features of a similar problem on the whole line
$\mathbb{R}$, which provides the model functions to construct solutions
of problem (\ref{eq:P}).

Let us consider 
\begin{equation}
-U''+U=U^{p}\text{ in }\mathbb{R},\ U>0.\label{eq:problimite}
\end{equation}
It is well known (see \cite{Kw}) that this equation admits a unique  -up to translations- solution  in $H^1(\R)$
 which  has the explicit form

\begin{equation}
U(x)=\left(\ensuremath{\frac{p+1}{2}}\right)^{\frac{2}{p-1}}\ensuremath{\left[\cosh\ensuremath{\left(\frac{p-1}{2}x\right)}\right]}^{-\frac{2}{p-1}}\,.\label{eq:bubble}
\end{equation}

We set $m_\infty:=\frac{1}{2}\left(\frac{1}{2}-\frac{1}{p+1}\right)\|U\|_{H^1(\R)}^2$ (see Section \ref{sec:Profile}). 

Notice that uniqueness of $H^1$ solution of (\ref{eq:problimite}) can be easily recovered when the equation 
is set in $\R^+$ (we refer to \cite{Kw} for the standard argument that can be straightforwardly adapted here). In this case any solution in $H^1(\R^+)$ has the form $U(x-x_0)\chi_{[0,+\infty)}$,
where $x_0$ is a suitable translation.

The next theorem gives, for a sufficiently large $\lambda$, a profile description of low action solutions of \eqref{eq:P}. In particular these solutions have a unique peak at a vertex of degree 1, they are similar to a suitable rescaling of $U$ on this edge and negligible in $L^\infty$ norm on the rest of the graph.
\begin{thm}
	\label{thm:profile}Let $\mathcal{G}$ be a compact graph with at least
	one terminal edge and $p>1$. Let $\lambda_{n}\rightarrow\infty$ and let, for
	any $n$, $u_{n}$ be a positive solution of (\ref{eq:P}) 
	with $\left.J_{\lambda_{n}}\right|_{\mathcal{N}_{\lambda_{n}}}(u_{n})\rightarrow m_{\infty}$.
	Then, up to subsequence, $u_{n}$ has a unique maximum point located in a terminal vertex $\vv$. 
	Moreover, denoting by $I=[0,l]$ the terminal edge where $u_n$ attains its maximum (with the convention that  the degree 1 vertex $\vv$ concides with $0$) we have that, while $n\rightarrow 0$
	\begin{enumerate}
	\item $u_n(\vv)\rightarrow +\infty$.
	\item $\lambda_{n}^{\frac{1}{1-p}}u_{n}\left(\frac{x}{\sqrt{\lambda_{n}}}\right)
\chi_l\left(\frac{x}{\sqrt{\lambda_{n}}}\right)\rightarrow U(x)$ weakly in $H^1(\R^+) $ and strongly 
 in $C^0(\R^+)$,
in $C^2_\text{loc}(\R^+)$ and in $L^t_\text{loc}(\R^+)$ for all $t\ge2$. Here $\chi_l$ is a cut off function.
	\item $\lambda_n^{\frac{1}{1-p}}\|u_{n}(x)-\lambda_n^{\frac{1}{p-1}}
	U(x\sqrt{\lambda_n})\|_{C^0([0,l/2])}\rightarrow0$	
	\item For every $l_1\in(0,l)$ and every $0<l_{1}<x\le l$, there exist two constants $c_{1},c_{2}>0$, depending on $l_1$ but independent from $n$, such that 
	\begin{align*}
	&u_{n}(x)  \le c_{1}\lambda_{n}^{\frac{1}{p-1}}e^{-c_{2}\sqrt{\lambda_{n}}x}\text{ on }[l_{1},l]\subset I\,,\\
	&\|u_{n}\|_{L^\infty(\G\setminus I)}  \le c_{1}\lambda_{n}^{\frac{1}{p-1}}e^{-c_{2}\sqrt{\lambda_{n}}l}\,.
	\end{align*} 
	\end{enumerate}
	\end{thm}
We point out that the assumption $\left.J_{\lambda_{n}}\right|_{\mathcal{N}_{\lambda_{n}}}(u_{n}) \to m_{\infty}$ is consistent, as the sets of solutions $u_n$ fulfilling it is actually not empty (see Section \ref{sec:Profile} and Corollary \ref{cor:limite}).

Reversing the perspective, whenever $\G$ has at least a vertex of degree 1 and again for large $\lambda$, it is possible to construct one-peaked
solutions to problem (\ref{eq:P}), using the function $U$ as a model,
and, if $\mathcal{G}$ has several terminal edges, then it is possible to
construct multipeaked solution. This is what is stated in the following two theorems.
\begin{thm}
	\label{thm:1picco}Let $\mathcal{G}$ be a compact graph with a vertex
	$\vv_{1}$ with degree $1$ and $p>1$. Denote by $I_1=[0,l_1]$ the terminal edge ending at $v_1$, with the convention that $v_1$ coincides with 0. Then, provided $\lambda$ is sufficiently large, 
	there exists a solution $u_{\lambda}$ of (\ref{eq:P}) with a single
	peak at $\vv_{1}$, i.e. $u_\lambda$ of the form
	\[
	u_\lambda:=W_\lambda+\phi\,, 
	\]
	with
	\[
	W_\lambda(x)=\chi(x) U_\lambda(x)
	\]
	where $\chi$ is a smooth cut--off function supported on $[0,l]\subset I_1$, for some $l<l_1$, and 
	\[
	U_\lambda(x)=\begin{cases}
	\lambda^{\frac{1}{p-1}}U(\lambda x) & \text{on }I_1\\
	0 & \text{on }\G\setminus I_1\,,
	\end{cases}
	\]
	$U$ being as in \eqref{eq:bubble}, and
	\[
	\|\phi\|_\lambda=O(\lambda^{-\alpha})
	\]
	for every $\alpha>0$. Furthermore, 
	\[
	\rho^{2}:=|u_{\lambda}|_{L^{2}(\mathcal{G})}^2=C\lambda^{\frac{5-p}{2(p-1)}}+l.o.t.
	\]
\end{thm}
\begin{thm}
	\label{thm:multi}
	Let $\mathcal{G}$ be a compact graph with $m\ge1$ vertices
	with degree $1$ and $p>1$. Choose $\vv_{1},\dots,\vv_{k}$ vertices of degree $1$
	with $1\le k\le m$. Let also $I_i=[0,l_i]$ denote the terminal edge ending at $v_i$, with the convention that $v_i$ coincides with 0.
	Then, provided $\lambda$ is sufficiently large, there
	exists a $k$-peaked solution $u_\lambda$ of (\ref{eq:P}) with a single
	peak at any vertex $\vv_{i}$, $i=1,\dots,k$, i.e. $u_\lambda$ of the form
	\[
	u_\lambda=W_\lambda+\phi\,,
	\]
	with
	\[
	W_\lambda(x)=\sum_{i=1}^k\chi_i(x)U_{\lambda,i}(x)
	\]
	where $\chi_i$ is a smooth cut--off function supported on $[0,l]\subset I_i$, for some $l<\min_{1\leq i\leq k}l_i$, and 
	\[
	U_{\lambda,i}(x)=\begin{cases}
	\lambda^{\frac{1}{p-1}}U(\lambda x) & \text{on }I_i\\
	0 & \text{on }\G\setminus I_i\,,
	\end{cases}
	\]
	$U$ being as in \eqref{eq:bubble}, and
	\[
	\|\phi\|_\lambda=O(\lambda^{-\alpha})
	\]
	for every $\alpha>0$. Furthermore, 
	\[
	\rho^{2}:=|u|_{L^{2}(\mathcal{G})}^2=C\lambda^{\frac{5-p}{2(p-1)}}+l.o.t.
	\]
\end{thm}
In Section \ref{sec:LS} we will show
this construction based on the Ljapunov--Schmidt finite dimensional
reduction. Again this procedure is possible for every $p$, and the
link between $\lambda$ and $\rho$ is explicit. This gives also
another interpretation of $L^{2}$-critical exponent $p=5$.

\begin{rem} A further observation about one-peaked solutions is possible. 
Given a sequence $\lambda_n\rightarrow +\infty$, let, $u_{\lambda_n}:=W_{\lambda_n}+\phi_n$ be the corresponding 
one-peaked solution obtained by Theorem \ref{thm:1picco} for any 
$\lambda_n$. The sequence 
$\{u_{\lambda_n}\}_n$ fulfills the hypothesis of Theorem \ref{thm:profile}, so it inherits all the properties 
given by Theorem \ref{thm:profile}: the exact location of the unique maximum point, the decreasing 
monotonicity, the decay rate and so on.
\end{rem}


\subsection{Notations}
Herafter we will use the following recurrent notations.
\begin{itemize}
\item $B_{P,r}=B(P,r)$ is the ball centered at $P$ with radius $r$. We use the same notation either if 
$B_{P,r}\subset\R$ or $B_{P,r}\subset\R^+$. In the last case, if $0\le P <r$ we intend $B_{P,r}=\{0\le x<P+ r\}$. 
Finally, $B_r:=B(0,r)$.
\item $\chi_\rho$ is a smooth cut--off function such that  $\chi_\rho=1$ when $x\in B_{\rho/2}$ and   $\chi_\rho=0$ outside a 
ball of radius $\rho$.
 When no ambiguity is possible we will omit the subscript $\rho$.
\item $\chi_{[0,+\infty)}$ is the characteristic function of $[0,+\infty)$.
\item With abuse of notation we often identify an edge $I\in \G$ with $[0,l]$, $l$ being the lenght of the edge. When the edge is a terminal one, the vertex $\vv$ of degree $1$ will be identified with $0$. 
\item Given a vertex $\vv\in \mathcal{G}$ we will suppose w.l.o.g. that the degree of that vertex is either $1$ or strictly larger than $2$. In fact, degree 2 vertices are indistinguishable from internal points. 
\end{itemize}

\section{Profile of low action solution\label{sec:Profile}}

As stated in the previous sections, for any $\lambda>0$ a solution of (\ref{eq:P}) can be obtained as a
critical point of the action functional $J_{\lambda}$ defined (see (\ref{Jlambda})) as
\begin{align*}
J_{\lambda}: & H^{1}(\mathcal{G})\rightarrow\mathbb{R}\\
J_{\lambda}(u)=\lambda^{\frac{1}{2}-\frac{p+1}{p-1}} & \int_{\mathcal{G}}\frac{\left(u'\right)^{2}}{2}+\frac{\lambda u^{2}}{2}-\frac{u^{p+1}}{p+1}dx
\end{align*}
on the Nehari manifold
\begin{align}
\mathcal{N}_{\lambda}: & =\left\{ u\in H^{1}(\mathcal{G})\smallsetminus\left\{ 0\right\} \ :\ J_{\lambda}'(u)[u]=0\right\} \nonumber \\
 & =\left\{ u\in H^{1}(\mathcal{G})\smallsetminus\left\{ 0\right\} \ :\ \|u\|_{\lambda}^{2}=|u|_{p+1}^{p+1}\right\} \label{eq:nehari}
\end{align}
It is standard to prove that $\mathcal{N}_{\lambda}$ is a $C^{1}$
manifold and that the Palais-Smale condition
holds on $\mathcal{N}_{\lambda}$. Moreover, by (\ref{eq:nehari})
we have that
\[
\left.J_{\lambda}\right|_{\mathcal{N}_{\lambda}}(u)=\lambda^{\frac{1}{2}-\frac{p+1}{p-1}}\left(\frac{1}{2}-\frac{1}{p+1}\right)\|u\|_{\lambda}^{2}.
\]
The Nehari manifold is not empty, in fact, problem (\ref{eq:P})
admits a constant solution. Also, any solution $u_{\lambda}$ that
we will find in Section \ref{sec:LS} belongs to $\mathcal{N}_{\lambda}$. 

One can easily prove that $\inf_{\mathcal{N}_{\lambda}}J_{\lambda}>0$
and, since Palais-Smale holds, that a non trivial minimizer exists.
We set 
\[
m_{\lambda}:=\inf_{\lambda}\left.J_{\lambda}\right|_{\mathcal{N}_{\lambda}}>0.
\]
The one peaked solution of Section \ref{sec:LS} allows also to estimate
$m_{\lambda}$ in term of the $H^{1}(\mathbb{R}^{+})$ norm of the
function $U$ defined in (\ref{eq:bubble}). Let us take $u_{\lambda}$
a one-peaked solution given by Theorem \ref{thm:1picco}. Let $I_{1}=[0,l_{1}]$
be the terminal edge where the peak is located, and suppose that the
terminal vertex is in $x=0$. We know that 
\[
u_{\lambda}=W_{\lambda}(x)+\phi
\]
where $W_{\lambda}(x)=\chi(x)U_{\lambda}(x)$, $\chi=1$ if $x\in I_{1}$
and $0\le x\le\delta$, $\chi=0$ if $x\in I_{1}$ and $2\delta\le x\le l_{1}$
for some fixed $\delta$ and 
\[
U_{\lambda}(x)=\left\{ \begin{array}{cc}
\lambda^{\frac{1}{p-1}}U(x\sqrt{\lambda}) & \text{ on }I_{1}\\
0 & \text{elsewhere}
\end{array}\right..
\]
Moreover $\|\phi\|_{\lambda}\le\lambda^{-\alpha}$ for any positive
$\alpha$. Thus we compute 
\[
\left.J_{\lambda}\right|_{\mathcal{N}_{\lambda}}(u_{\lambda})=\lambda^{\frac{1}{2}-\frac{p+1}{p-1}}\left[\left(\frac{1}{2}-\frac{1}{p+1}\right)\|U_{\lambda}\|_{\lambda}^{2}\right]+o(1)=\frac{1}{2}\left(\frac{1}{2}-\frac{1}{p+1}\right)\|U\|_{H^{1}(\mathbb{R})}^{2}+o(1).
\]
 Set 
\begin{eqnarray*}
m_{\infty}:=\frac{1}{2}\left(\frac{1}{2}-\frac{1}{p+1}\right)\|U\|_{H^{1}(\mathbb{R})}^{2} &  & m_{\lambda}:=\inf_{\lambda}\left.J_{\lambda}\right|_{\mathcal{N}_{\lambda}}
\end{eqnarray*}
and we get 
\begin{equation}
0\le\lim_{\lambda\rightarrow\infty}m_{\lambda}\le m_{\infty}.\label{eq:dallalto}
\end{equation}
This proves, also, that it is possible to find a sequence $\{u_n\}_n$ fulfilling the hypothesis of 
Theorem \ref{thm:profile}. We are able, by proving this theorem, to give an asymptotic profile 
description for a positive low action solution of problem (\ref{eq:P}).
 
\begin{proof}[Proof of Theorem \ref{thm:profile}]
The proof is divided in several steps.

\medskip
\noindent \emph{Step 1: For $n$ large $u_{n}$ is not constant.}

Indeed, if $u_{n}\equiv C$, then, by (\ref{eq:P}) necessarily $C=\sqrt[p-1]{\lambda_{n}}.$
Then 
\begin{align*}
\left.J_{\lambda_{n}}\right|_{\mathcal{N}_{\lambda_{n}}}(u_{n}) & =\lambda_{n}^{\frac{1}{2}-\frac{p+1}{p-1}}\left[\left(\frac{1}{2}-\frac{1}{p+1}\right)\|u_n\|_{\lambda}^{2}\right]=\lambda_{n}^{\frac{1}{2}-\frac{p+1}{p-1}}\left[\left(\frac{1}{2}-\frac{1}{p+1}\right)\lambda_{n}^{\frac{p+1}{p-1}}\left|\mathcal{G}\right|\right]\\
 & =\lambda_{n}^{\frac{1}{2}}\left[\left(\frac{1}{2}-\frac{1}{p+1}\right)\left|\mathcal{G}\right|\right]\rightarrow\infty\text{ for }\lambda_{n}\rightarrow\infty,
\end{align*}
where $\left|\mathcal{G}\right|=\int_{\mathcal{G}}1dx$ is the length
of the graph. This contradicts $\left.J_{\lambda_{n}}\right|_{\mathcal{N}_{\lambda_{n}}}(u_{n})\to m_{\infty}$.

\medskip
\noindent \emph{Step 2: $u_{n}$ has a maximum point $P_{n}$. Moreover, $u_{n}(P_{n})\ge\sqrt[p-1]{\lambda}_{n}$.}

First, by standard regularity theory, we have that $u_{n}$ is a regular
solution, that is, for any edge $I\subset\mathcal{G}$, $\left.u_{n}\right|_{I}\in C^{2}(\bar{I})$.
Since $u_{n}$ is not constant, and the graph is compact, $u_{n}$
has a global maximum point $P_{n}\in\mathcal{G}$. 

Now, if $P_{n}$ is on the interior of some edge $I$, we have that
$u_{n}'(P_{n})=0$ and $u_{n}''(P_{n})\le0$. Thus, by (\ref{eq:P})
we get $\lambda u_{n}(P_{n})-u_{n}^{p}(P_{n})=u_{n}''(P_{n})\le0$,
so $u_{n}(P_{n})\ge\sqrt[p-1]{\lambda}_{n}$. 

If $P_{n}$ is assumed on a terminal vertex, again we have $u_{n}'(P_{n})=0$
by Kirchhoff condition, so necessarily we have $u_{n}''(P_{n})\le0.$
Thus again $u_{n}(P_{n})\ge\sqrt[p-1]{\lambda}_{n}$. 

Finally suppose that $P_{n}$ is on a vertex of degree greater than
1. Since $P_{n}$ is a maximum point, $\frac{d\left(u_{n}\right)_{e}}{dx}(P_{n})\le0$
on any edge $e$ that insists on the vertex. Since, by (\ref{eq:P}), 
$\sum_{e\prec P_{n}}\frac{d(u_{n})_{e}}{dx}(P_{n})=0$, we have $\frac{d\left(u_{n}\right)_{e}}{dx}(P_{n})=0.$
At this point there exists at least an edge $e\prec P_{n}$ for which
$(u_{n})_{e}''(P_{n})\le0$ and we conclude as before.

\medskip
\noindent \emph{Step 3: There exists a vertex $\vv\in \mathcal{G}$ such that, up to 
subsequences, $d(P_{n},\vv)\rightarrow 0$ while $n\rightarrow\infty$.}

Suppose, by contradiction, that $\lim_n\inf_{\vv\in \mathcal{G}}d(P_{n},\vv)=\delta> 0$.
Up to subsequences we can suppose that $P_n\in I$ for all $n$ and we can identify $I=[0,l]$ $\vv=0$.
Thus we define 
\[
v_{n}(x):=\lambda_{n}^{\frac{1}{1-p}}u_{n}\left(\frac{x}{\sqrt{\lambda_{n}}}+P_{n}\right)
\chi_\delta\left(\frac{x}{\sqrt{\lambda_{n}}}+P_{n}\right)\text{ for }|x/\sqrt{\lambda_{n}}|\le \delta.
\]
The function $v_{n}$ belongs
to $H^{1}(\mathbb{R})$, moreover
\begin{align*}
\|v_{n}\|_{H^{1}(\mathbb{R})}^{2} & \le C\lambda_{n}^{\frac{2}{1-p}}
\int_{B_{\delta\sqrt{\lambda_{n}}}}\left[\frac{d}{dx}u_{n}\left(\frac{x}{\sqrt{\lambda_{n}}}+P_{n}\right)\right]^{2}+\left[u_{n}\left(\frac{x}{\sqrt{\lambda_{n}}}+P_{n}\right)\right]^{2}dx\\
 & =C\lambda_{n}^{\frac{2}{1-p}}
 \int_{B_{\delta\sqrt{\lambda_{n}}}}\frac{1}{\lambda_{n}}\left(u_{n}'\right)^{2}\left(\frac{x}{\sqrt{\lambda_{n}}}+P_{n}\right)+u_{n}^{2}\left(\frac{x}{\sqrt{\lambda_{n}}}+P_{n}\right)dx\\
 & =C\lambda_{n}^{\frac{p+1}{1-p}}\sqrt{\lambda_{n}}
 \int_{B_{P_{n},\delta}}\left(u_{n}'\right)^{2}\left(x\right)+\lambda_{n}u_{n}^{2}\left(x\right)dx\\
 & \le C\lambda_{n}^{\frac{p+1}{1-p}}\sqrt{\lambda_{n}}\int_{I}\left(u_{n}'\right)^{2}\left(x+P_{n}\right)+\lambda_{n}u_{n}^{2}\left(x+P_{n}\right)dx\\
 & \le C\lambda_{n}^{\frac{p+1}{1-p}}\sqrt{\lambda_{n}}\|u_{n}\|_{\lambda}^{2}\le C\left(\frac{p-1}{2(p+1)}\right)\left.J_{\lambda_{n}}\right|_{\mathcal{N}_{\lambda_{n}}}(u_{n})\le Cm_{\infty.}
\end{align*}
So $\left\{ v_{n}\right\} _{n}$ is bounded in $H^{1}(\mathbb{R})$,
hence there exists $v\in H^{1}(\mathbb{R})$ such that $v_{n}\rightharpoonup v$
weakly in $H^{1}(\mathbb{R})$ and $v_{n}\rightarrow v$ strongly
in $L_{\text{loc}}^{t}(\mathbb{R})$  for any $t\ge2$ and in $C_{\text{loc}}^{0}(\mathbb{R})$.
We want to
prove that $v$ is a nontrivial solution of (\ref{eq:problimite}). 

Take $\varphi\in C_{0}^{\infty}(\mathbb{R})$. For $n$ large we have
that the support $\text{spt}(\varphi)$ of $\varphi$ is contained
in $B_{\frac \delta2\sqrt{\lambda_{n}}}$. We define a sequence of function
$\left\{ \varphi_{n}\right\} _{n}\in H^{1}(\mathcal{G})$ (for $n$
large) as
\[
\varphi_{n}(x)=\left\{ \begin{array}{cc}
\lambda_{n}^{\frac{1}{p-1}}\varphi\left(\sqrt{\lambda_{n}}(x-P_{n})\right) & \text{ on }I\\
0 & \text{elsewhere}\,.
\end{array}\right.
\]
Since $u_{n}$ is a solution of (\ref{eq:P}) we have
\begin{align*}
0= & J'_{\lambda_{n}}(u_{n})[\varphi_{n}]=\int_{I}u_{n}'\varphi_{n}'+\lambda_{n}u_{n}\varphi_{n}-u_{n}^{p}\varphi_{n}dx\\
= & \lambda_{n}^{\frac{2}{p-1}}\int_{I}\frac{d}{dx}v_{n}\left(\sqrt{\lambda_{n}}(x-P_{n})\right)\frac{d}{dx}\varphi\left(\sqrt{\lambda_{n}}(x-P_{n})\right)dx\\
 & +\lambda_{n}^{\frac{2}{p-1}}\lambda_{n}\int_{I}v_{n}\left(\sqrt{\lambda_{n}}(x-P_{n})\right)\varphi\left(\sqrt{\lambda_{n}}(x-P_{n})\right)dx\\
 & -\lambda_{n}^{\frac{p+1}{p-1}}\int_{I}v_{n}\left(\sqrt{\lambda_{n}}(x-P_{n})\right)\varphi\left(\sqrt{\lambda_{n}}(x-P_{n})\right)dx\\
= & \lambda_{n}^{\frac{p+1}{p-1}-\frac{1}{2}}\int_{\mathbb{R}}v_{n}'\varphi'+v_{n}\varphi-v_{n}^{p}\varphi dx,
\end{align*}
so by weak convergence on $H^{1}(\mathbb{R})$ 
\[
\int_{\mathbb{R}}v'\varphi'+v\varphi-v^{p}\varphi=0\text{ for any }\varphi\in C_{0}^{\infty}(\mathbb{R}).
\]
Since, by Step 2, $u_{n}(P_{n})\ge\lambda_{n}^{\frac{1}{p-1}}$ then
$v_{n}(0)=\lambda_{n}^{\frac{1}{1-p}}u_{n}(P_{n})\ge1$, so by $L_{\text{loc}}^{t}$
convergence we can prove that $v\neq0$. Thus, by uniqueness of solutions of (\ref{eq:problimite}) 
we have that $v=U$. This leads to
a contradiction. In fact, there exists $R>0$ such that 
\[
|U|_{L^{p+1}(B_{R})}^{p+1}>\frac{3}{4}|U|_{L^{p+1}(\mathbb{R})}^{p+1}
\]
and, since $v_{n}\rightarrow v=U$ in $L_{\text{loc}}^{p+1}$ there
exists $n_{0}>1$ such that
\[
|v_{n}|_{L^{p+1}(B_{R})}^{p+1}>\frac{3}{4}|U|_{L^{p+1}(\mathbb{R})}^{p+1}\text{ for }n>n_{0}.
\]
On the other hand, there exists $n_{1}>1$ such that, for $n>n_{1}$ it holds
$R/\sqrt{\lambda_{n}}<\delta/2$, so that if $|x|\le R$ then 
$x/\sqrt{\lambda_{n}}+P_{n}\in B_{P_{n},\frac \delta 2}$
and $\chi(x)\equiv1$. So, for $n$ large we have
\begin{align*}
|v_{n}|_{L^{p+1}(B_{R})}^{p+1} & \le\lambda_{n}^{-\frac{p+1}{p-1}}\int_{B_{R}}|u_{n}|^{p+1}(x/\sqrt{\lambda_{n}}+P_{n})dx\le\lambda_{n}^{\frac{1}{2}-\frac{p+1}{p-1}}\int_{B_{P_{n},\delta}}|u_{n}|^{p+1}dx\\
 & \le\lambda_{n}^{\frac{1}{2}-\frac{p+1}{p-1}}|u_{n}|_{L^{p+1}(\mathcal{G})}^{p+1}.
\end{align*}
So 
\begin{align}
\left.J_{\lambda_{n}}\right|_{\mathcal{N}_{\lambda_{n}}}(u_{n}) & =\lambda_{n}^{\frac{1}{2}-\frac{p+1}{p-1}}\left[\left(\frac{1}{2}-\frac{1}{p+1}\right)|u_{n}|_{L^{p+1}(\mathcal{G})}^{p+1}\right]\geq\left[\left(\frac{1}{2}-\frac{1}{p+1}\right)|v_{n}|_{L^{p+1}(B_{R})}^{p+1}\right]\nonumber \\
 & >\frac{3}{4}\left(\frac{1}{2}-\frac{1}{p+1}\right)|U|_{L^{p+1}(\mathbb{R})}^{p+1}=\frac{3}{4}\left(\frac{1}{2}-\frac{1}{p+1}\right)\|U\|_{H^{1}(\mathbb{R})}^{2}=\frac{3}{2}m_{\infty}\label{eq:StimaJn}
\end{align}
that contradicts our assumption, thus implying $\lim_n\inf_{\vv\in\G} d(P_n,\vv)=0$.

\medskip
\noindent \emph{Step 4: Given $\vv$ as in the previous step, we have 
$\lim_n d(P_n,\vv)\sqrt{\lambda_n}=0$.}

Suppose, by contradiction, that $\lim_n d(P_n,\vv)\sqrt{\lambda_n}=\delta >0$. Define

\begin{equation}
\label{eq:w_n}
w_{n}(x):=\lambda_{n}^{\frac{1}{1-p}}u_{n}\left(\frac{x}{\sqrt{\lambda_{n}}}\right)
\chi_l\left(\frac{x}{\sqrt{\lambda_{n}}}\right)\text{ for }0\le x/\sqrt{\lambda_{n}}\le l.
\end{equation}
The function $w_{n}$ belongs
to $H^{1}(\mathbb{R^+})$, and, in analogy with Step 3, we can prove that $w_{n}\rightharpoonup w$
weakly in $H^{1}(\mathbb{R}^+)$ and $w_{n}\rightarrow w$ strongly
in $L_{\text{loc}}^{t}(\mathbb{R}^+)$  for any $t\ge2$ and in  $C_{\text{loc}}^{0}(\mathbb{R}^+)$.
Given $\varphi\in C_{0}^{\infty}((0,+\infty))$, for $n$ large we have
that the support $\text{spt}(\varphi)$ of $\varphi$ is contained
in $B_{\frac l2\sqrt{\lambda_{n}}}$ and we can prove, as before, that 
$w$ is a nontrivial positive solution of (\ref{eq:problimite}) on $\R^+$, 
although we do not know its value at the origin. By uniqueness of solutions of (\ref{eq:problimite}) on $\R^+$, 
we have that $w=U(x-x_0)\chi_{[0,\infty)}$ for some suitable $x_0\in\R$. Since for the maximum point of 
$u_n$ it holds $P_n\sqrt\lambda_n\ge \delta/2>0$, we have that $w$ has a maximum point in $(0,+\infty)$, so $x_0>0$. At this point we can prove, similarly to Step 3, that there exists $K>1$ such that 
$$
\left.J_{\lambda_{n}}\right|_{\mathcal{N}_{\lambda_{n}}}(u_{n}) >Km_\infty
$$
which contradicts our hypothesis.

\medskip
\noindent \emph{Step 5: $\vv$ coincides with an extremal vertex.}

Suppose, by contradiction that $\vv$ is a vertex with degree $k\ge 3$.

To simplify the notation, let $I_{1}=[0,l_{1}],\dots, I_{k}=[0,l_{k}]$
the edges that intersect in $\vv$ and let us suppose that
for any $I_{j}$, coordinates $x_j$ are defined on $I_j$ so that $\vv$ coincides with $x_j=0$, as shown in Figure \ref{figure}.
Suppose, also, that $P_n\in I_1$.
\begin{figure}[t]
\centering
\includegraphics[scale=0.6]{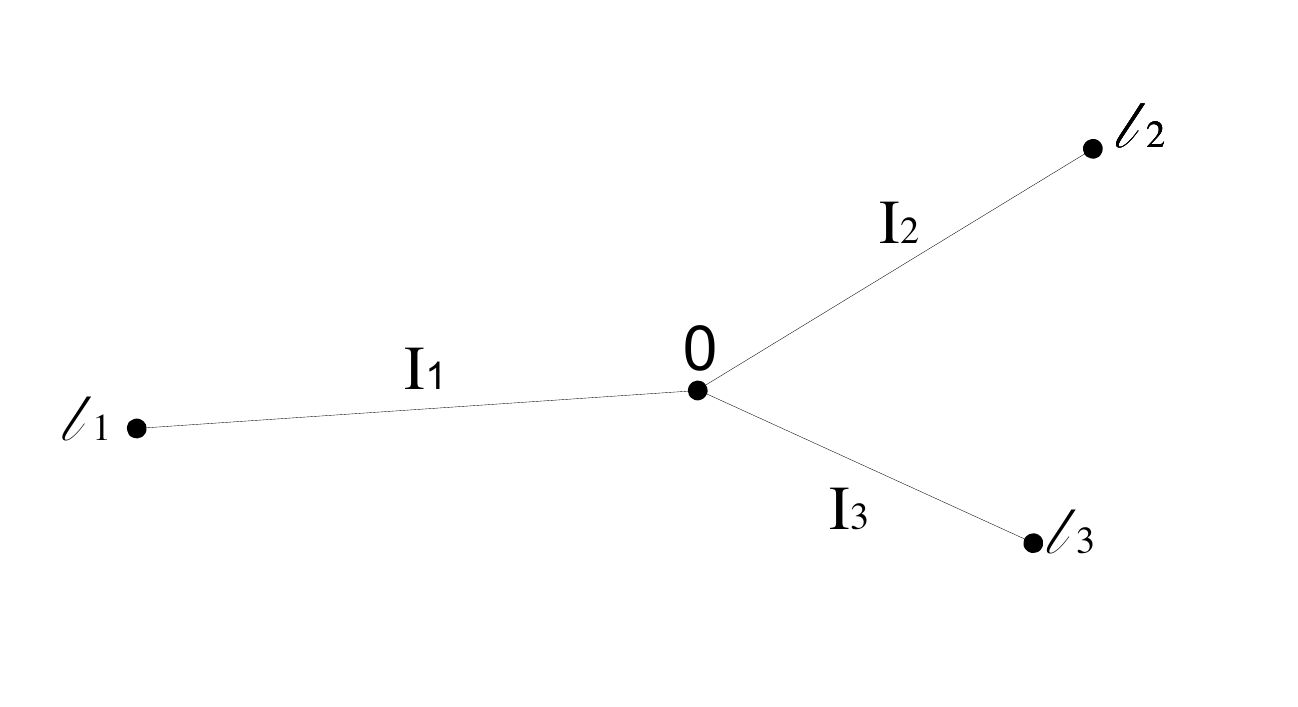}
\caption{example of a labelling of the edges entering a vertex $\vv$ as in Step 5 of the proof of Theorem \ref{thm:profile}.}
\label{figure}
\end{figure}

Choose $\rho<\min_{k}l_{k}$ and define, for $j=1,\dots,k$, $u_n^j:=\left.u_n\right|_{I_j}$ and
\[ 
v_{n}^{j}(x):=\lambda_{n}^{\frac{1}{1-p}}u_{n}^j\left(\frac{x}{\sqrt{\lambda_{n}}}\right)
\chi_\rho\left(\frac{x}{\sqrt{\lambda_{n}}}\right)\text{ for }0\le x/\sqrt{\lambda_{n}}\le\rho.
\]
As before, for any $j$, $\{v_n^j\}_n$ is bounded in $H^1(\R^+)$, and converges to some 
$v^j$ weakly in $H^{1}(\mathbb{R}^+)$ and strongly 
in $L_{\text{loc}}^{t}(\mathbb{R}^+)$ for any $t\ge2$ and in  $C_{\text{loc}}^{0}(\mathbb{R}^+)$.

Given any $R>0$, there exists $n$ sufficiently large such that $R<\rho\sqrt{\lambda_n}/2$, so on $[0,R]$ 
we have that $v^j_n(x)\equiv\lambda_{n}^{\frac{1}{1-p}}u_{n}^j\left(\frac{x}{\sqrt{\lambda_{n}}}\right)$.
Now, since $u_n$ solves (\ref{eq:P}), we have that 
$$
(v_n^j)''=v_n^j-(v_n^j)^p \text{ on }[0,R] 
$$
and, since $v_n^j\rightarrow v^j$ in $C^0([0,R])$, and by the arbitrariness of $R$ we have that 
$v_n^j\rightarrow v^j$ in   $C_{\text{loc}}^{0}(\mathbb{R}^+)$ for all $j$.

Finally $v^j$ is a nontrivial positive solution of (\ref{eq:problimite}) on $\R^+$, so 
$$
v^j(x)=U(x-x_j)\chi_{[0,+\infty)}\text{ for some }x_j\in\R.
$$
We can prove that $x_j=0$ for all $j$. In fact, we have that $P_n$ is a maximum point for $u_n$, so 
$P_n\sqrt{\lambda_n}$ is a maximum point for $v_n^1$, so $(v_n^1)'(P_n\sqrt{\lambda_n})=0$. Since,
by Step 4,  $P_n\sqrt{\lambda_n}\rightarrow0$ , we have that $(v^1)'(0)=0$ for $C^2$ convergence. Thus 
$x_1=0$. Moreover $u_n^j(0)=u_n^1(0)$ for any $j$ by continuity of $u_n$. Then also 
 $v_n^j(0)=v_n^1(0)$  and, passing to the limit in $n$, also that 
 $v^j(0)=v^1(0)$ for any $j$. Thus $x_j=0$ for all $j$, since $U$ has a unique maximum.
At this point, proceeding as before we have 
\[
\left.J_{\lambda_{n}}\right|_{\mathcal{N}_{\lambda_{n}}}(u_{n})>\frac{3k}{2}m_{\infty}+o(1)
>m_{\infty}
\]
which leads us to a contradiction.

\medskip
\noindent \emph{Step 6: $u_n$ has a unique maximum. Moreover, this maximum coincides with $\vv$.}

By contradiction, suppose that $u_n$ has another maximum point $Q_n\neq P_n$. By the previous 
step, up to subsequences, it is possible to prove that there exists a terminal vertex $\textsc{w}$ in 
$\mathcal{G}$ such that $\lim_n d(Q_n,\textsc{w})\sqrt{\lambda_n}=0$. 
Moreover, one can check that $\textsc{w}$ must coincide with $\vv$, otherwise 
$\left.J_{\lambda_{n}}\right|_{\mathcal{N}}>\frac32 m_{\infty}$.

Thus  $\lim_n d(Q_n,\vv)\sqrt{\lambda_n}=0$.
At this point, let $w_n$ be as in \eqref{eq:w_n} in Step 4
\[ 
w_{n}(x)=\lambda_{n}^{\frac{1}{1-p}}u_{n}\left(\frac{x}{\sqrt{\lambda_{n}}}\right)
\chi_l\left(\frac{x}{\sqrt{\lambda_{n}}}\right)\text{ for }0\le x/\sqrt{\lambda_{n}}\le l
\]
and, setting $p_n=P_n\sqrt{\lambda_n}$, $q_n=Q_n\sqrt{\lambda_n}$,
we have
\begin{equation}\label{pn}
p_n,q_n\rightarrow 0\text{ while }n\rightarrow \infty\text{, and }
w_n'(p_n)=w_n'(q_n)=0.
\end{equation}
By the previous steps it holds $w_n(x)\rightarrow U(x)\chi_{[0,+\infty)}$ in $C^2_\text{loc}(\R^+)$, 
thus implying $w''(0)<0$. On the other hand, in light of (\ref{pn}) we have $w''(0)=0$ which gives us a 
contradiction.

We can prove that $P_n\equiv \vv$ exactly with the same argument, using the fact that $u_n'(\vv)=0$ since $u_n$ solves (\ref{eq:P}).

\medskip
\noindent \emph{Step 7: $w_n(x)\rightarrow U(x)\chi_{[0,+\infty)}$ in $C^0(\R^+)$.}

By Step 6, we have that $w_n$ is decreasing on $[P_n\sqrt{\lambda_n},+\infty)$. Now, given $\varepsilon>0$
there exists an $R=R(\varepsilon)$ such that $U(R)\le \varepsilon/4$. Moreover, there 
exists $\bar n=\bar n(R)$ such that, for $n>\bar n$, $ \|w_n-U\|_{C^0([0,R])}\le \varepsilon/4$. So
\begin{multline}
\|w_n-U\|_{C^0(\R^+)}\le \|w_n-U\|_{C^0([0,R])}+\|w_n\|_{C^0([R,+\infty))}+\|U\|_{C^0([R,+\infty))}\\
\le \|w_n-U\|_{C^0([0,R])}+w_n(R)+U(R)\\
+\le \|w_n-U\|_{C^0([0,R])}+|w_n(R)-U(R)|+2U(R)\\
\le  2\|w_n-U\|_{C^0([0,R])}+2U(R)\le \varepsilon.
\end{multline}

\medskip
\noindent \emph{Step 8: Proof of Claim 1--2--3.}

The proof of Claims 1 and 2 of the Theorem is a direct consequence of the previous steps.
Moreover by Step 7
$$\left\|\lambda_{n}^{\frac{1}{1-p}}\left.u_{n}\right|_I\left(\frac{x}{\sqrt{\lambda_{n}}}\right)
- U(x)\right\|_{C^0([0,l\sqrt{\lambda_n}/2])}\rightarrow 0$$
and by a change of variable we obtain Claim 3. 

\medskip
\noindent\emph{Step 9: proof of Claim 4.}

Let $l_1\in(0,l)$ be given. First, we can repeat the argument of the previous steps to prove that
$u_{n}$ has no local maximum point except for the extremal vertex.
Therefore, $u_{n}$ is strictly decreasing on any edge of the graph. 

Given again as in \eqref{eq:w_n}

\[
w_{n}(x):=\lambda_{n}^{\frac{1}{1-p}}u_{n}\left(\frac{x}{\sqrt{\lambda_{n}}}\right)\chi_{l}\left(\frac{x}{\sqrt{\lambda_{n}}}\right)\text{ for }x\ge0,
\]
by Step 7 we have $w_{n}(x)\rightarrow U(x)\chi_{[0,+\infty)}$ in $C_{\text{loc}}^{2}(\mathbb{R}^{+})$
and, by definition, there exists a constant $C_{0}$ for which 
\[
U\le C_{0}e^{-x}\text{ for }x>0.
\]
Now, fix $0<\varepsilon<1/4$ and choose $R=2\log(C_{0}/\varepsilon)$.
Then there exists $\bar{n}=\bar{n}(\varepsilon)$ such that 
\[
\|w_{n}-U\|_{C^{2}[0,R]}\le\varepsilon\text{ for }n\ge\bar{n}.
\]
We have that 
\begin{equation}
w_{n}(x)\le2\varepsilon\text{ on }R/2\le x\le R,\label{eq:wn<eps}
\end{equation}
indeed 
\[
w_{n}(x)\le U(x)+\varepsilon\le C_{0}e^{-R/2}+\varepsilon\le2\varepsilon.
\]
Now (\ref{eq:wn<eps}) implies, by rescaling, that 
\[
u_{n}\left(\frac{x}{\sqrt{\lambda_{n}}}\right)\chi_{l}\left(\frac{x}{\sqrt{\lambda_{n}}}\right)\le2\lambda_n^{\frac{1}{p-1}}\varepsilon\text{ on }R/2\le x\le R,
\]
so that, since $\frac{R}{\sqrt{\lambda_{n}}}\le\frac{l}{2}$ for $n$ large,
we have
\[
u_{n}\left(y\right)\le2\lambda_n^{\frac{1}{p-1}}\varepsilon\text{ on }\frac{R}{2\sqrt{\lambda_{n}}}\le y\le\frac{R}{\sqrt{\lambda_{n}}},
\]
and, $u_{n}$ being strictly decreasing, 
\begin{equation}
u_{n}\left(y\right)\le2\lambda_{n}^{\frac{1}{p-1}}\varepsilon\text{ on }\frac{R}{2\sqrt{\lambda_{n}}}\le y\le l.\label{eq:4.33}
\end{equation}
Now, $u_{n}$ solves 
\[
u_n''-\left(\lambda_{n}-u_n^{p-1}\right)u_n=0\text{ on }0\le y\le l
\]
and, by (\ref{eq:4.33}) and since $\varepsilon\le1/4$, there exists
$a>0$ independent from $n$ such that 
\[
\lambda_{n}-u_n^{p-1}\ge\lambda_{n}\left(1-(2\varepsilon)^{p-1}\right)\ge a\lambda_{n}\text{ on }\frac{R}{2\sqrt{\lambda_{n}}}\le y\le l.
\]
Since it is well--known (see Lemma 2.4 of \cite{Fife}) that, whenever 
\[
u''-\lambda_n q(x)u=0\text{ on }0<l_1\leq x\leq l,\qquad q\geq a,
\]
there exist two constant $c_1,c_2>0$, independent of $\lambda_n$, such that
\[
u(x)\leq c_1\lambda_n^{\frac{1}{p-1}}e^{-c_2\sqrt{\lambda_n}x}
\]
for every $l_1\leq x\leq l$, we conclude.
\end{proof}
\begin{cor}
\label{cor:limite}We have 
\[
\lim_{\lambda\rightarrow\infty}m_{\lambda}=m_{\infty}.
\]
\end{cor}
\begin{proof}
By (\ref{eq:dallalto}) we have $\lim_{\lambda\rightarrow\infty}m_{\lambda}\le m_{\infty}$.
To prove the reverse inequality, assume by contradiction that there exists a sequence $\{u_n\}_n$ of solutions with
$\lim_n\left.J_{\lambda_{n}}\right|_{\mathcal{N}_{\lambda_{n}}}(u_{n})< m_{\infty}$, and, as in the 
proof of the previous theorem, we have
\[
w_{n}\rightarrow U\text{ in }L^{p+1}_\text{loc}(\mathbb{R}^{+})\,,
\]
$w_n$ being given by \eqref{eq:w_n}.
Now, for any $\eta$, there exists an $R=R(\eta)>0$ such that 
\[
|U|_{L^{p+1}([0,R])}^{p+1}>(1-\eta)|U|_{L^{p+1}(\mathbb{R}^{+})}^{p+1}
\]
and, since $w_{n}\rightarrow U$ in $L^{p+1}([0,R])$
there exists $n_{0}>1$ such that
\[
|w_{n}|_{L^{p+1}([0,R])}^{p+1}>(1-2\eta)|U|_{L^{p+1}(\mathbb{R}^{+})}^{p+1}\text{ for }n>n_{0}.
\]
At this point we can proceed similarly to (\ref{eq:StimaJn}), obtaining
\[
\left.J_{\lambda_{n}}\right|_{\mathcal{N}_{\lambda_{n}}}(u_{n})\ge(1-2\eta)m_{\infty}.
\]
The arbitrariness of $\eta$ provides the contradiction. 
\end{proof}

\section{Construction of peaked solutions\label{sec:LS}}

In order to perform the finite dimensional reduction, we have to linearize
Problem (\ref{eq:problimite}) around the solution $U$ and to study
the null space of the linearized problem, that is the set of solutions
to the Neumann boundary value problem 
\begin{equation}
\left\{ \begin{array}{cc}
-\psi''+\psi=pU^{p-1}\psi & \text{ in }\mathbb{R}^{+}\\
\psi'(0)=0.
\end{array}\right.\label{eq:linearizzato}
\end{equation}
While the equation $-\psi''+\psi=pU^{p-1}\psi$ in $\mathbb{R}$ has
a one-dimensional space of solutions generated by $Z(t)=U'(t)$, it
is easy to show that problem (\ref{eq:linearizzato}) has only the
trivial solution, due to the boundary condition. 

This result can be expressed in a more general form for the so called
\emph{star graphs}, i.e. graphs that are union of $n$ half lines all
connected to a same vertex $\vv_{0}$. If $\mathcal{G}$ is a star graph,
the function 
\[
\bar{u}=(u_{e})_{e=1,\dots,n}\text{ \text{where }}u_{e}(x)=U(x),\ x\ge0
\]
is a solution of problem (\ref{eq:P}). Linearizing (\ref{eq:P})
around this solution we get the Kirchhoff boundary value problem 
\begin{equation}
\left\{ \begin{array}{cc}
-\psi''+\psi=p\bar{u}^{p-1}\psi & \text{ in }\mathcal{G}\\
\sum_{e=1}^{n}\frac{d\psi_{e}}{dx}(\vv_{0})=0
\end{array}\right.\label{eq:starlinearizzato}
\end{equation}
and we can completely describe the space of $H^{1}(\mathcal{G})$
solutions of (\ref{eq:starlinearizzato}). 

We start looking for solutions of $-\psi''+\psi=pU^{p-1}\psi$ in $\mathbb{R}^{+}$
and then we will consider the boundary conditions. The space of solutions
in $H^{1}(\mathbb{R}^+)$ of $-\psi''+\psi=pU^{p-1}\psi$ is
spanned by the solutions of the following two boundary value problems
\begin{equation}
\left\{ \begin{array}{cc}
-\psi_{1}''+\psi_{1}=pU^{p-1}\psi_{1} & \text{ in }\mathbb{R}^{+}\\
\psi_{1}(0)=0
\end{array}\right.\label{eq:L1}
\end{equation}
and 
\begin{equation}
\left\{ \begin{array}{cc}
-\psi_{2}''+\psi_{2}=pU^{p-1}\psi_{2} & \text{ in }\mathbb{R}^{+}\\
\psi_{2}'(0)=0
\end{array}\right.\label{eq:L2}
\end{equation}
We can extend -respectively by odd or even reflection- any solution
of (\ref{eq:L1}) and (\ref{eq:L2}) to a solution of $-\psi''+\psi=pU^{p-1}\psi$
in $\mathbb{R}$. So we have that (\ref{eq:L2}) has no solution in
$H^{1}(\mathbb{R})$ and $\psi_{1}(x)=cU'(x)$. At this point a solution
of (\ref{eq:starlinearizzato}), taking into account the Kirchhoff boundary
condition is 
\[
\begin{array}{c}
\psi=(\psi_{e})_{e=1,\dots,n}\text{ with}\\
\psi_{e}(x)=c_{e}U'(x)\text{ for }x\ge0;\\
\sum_{e=1}^{n}c_{e}=0.
\end{array}
\]
This implies that the solution of (\ref{eq:starlinearizzato}) form
a $n-1$ dimensional linear space. This is in accordance to the case
$n=1,$ that is the half line, in which there are no solutions, and
$n=2,$ equivalent to $\mathbb{R},$ for which the linear space is
spanned by $U'(x)$.

\begin{rem}
	When dealing with the time--dependent NLS equation
	\[
	i\partial_t\psi(x,t)=-\Delta_x\psi(x,t)-|\psi(x,t)|^{p-1}\psi(x,t)+\psi(x,t)
	\]
	it is well--known that linearizing around a solution $U$ leads to the following system of equations
	\[
	\begin{pmatrix}
	L_+ & 0\\
	0 & L_-
	\end{pmatrix}
	\begin{pmatrix}
	\psi_1\\
	\psi_2
	\end{pmatrix}=0
	\]
	where $\psi(x,t)=\psi_1(x,t)+i\psi_2(x,t)$ and 
	
	\[
	\begin{split}
	&L_+:=-\Delta_x+1-pU^{p-1}\\
	&L_-:=-\Delta_x+1-U^{p-1}\,.
	\end{split}
	\]
	Note that the equation given by $L_+$ for the real part $\psi_1$ coincides with the one we derived in our discussion of the linearized problem. For the purposes of the forthcoming analysis, it is sufficient here to consider only this equation
	(for a wider discussion of the linearized problem on star graphs see for instance \cite{KP-JDE}.)
\end{rem} 

Coming back to our original problem,let us consider, for a given 
compact graph $\mathcal{G}$, the compact immersion
\[
i_{\lambda}:\left(H^{1}(\mathcal{G}),\left\langle ,\right\rangle _{\lambda}\right)\rightarrow\left(L^{2}(\mathcal{G}),\left\langle ,\right\rangle _{L^{2}}\right)
\]
and define its adjoint map
\[
i_{\lambda}^{*}:\left(L^{2}(\mathcal{G}),\left\langle ,\right\rangle _{L^{2}}\right)\rightarrow\left(H^{1}(\mathcal{G}),\left\langle ,\right\rangle _{\lambda}\right)
\]
such that 
\[
\left\langle i_{\lambda}^{*}(f),v\right\rangle _{\lambda}=\left\langle f,v\right\rangle _{L^{2}}\text{ for all }v\in H^{1}(\mathcal{G})\,,
\]
or equivalently 
\[
u=i_{\lambda}^{*}(f)\Leftrightarrow u\text{ solves }\left\{ \begin{array}{cc}
-u''+\lambda u=f & \text{ in }\mathcal{G}\\
\sum_{e\prec \vv}\frac{du_{e}}{dx}(\vv)=0 & \forall \vv\in V
\end{array}\right..
\]

\subsection{One peaked solutions}

We construct now a model profile for a solution which has a peak on
the extremal vertex $\vv_{1}$ (the vertex of degree 1) of the first
edge $I_{1}=[0,l_{1}]$. We suppose, without loss of generality that
$\vv_{1}$ corresponds to the coordinate $x=0$. We define 
\[
U_{\lambda}(x)=\left\{ \begin{array}{cc}
\lambda^{\frac{1}{p-1}}U\left(\sqrt{\lambda}x\right) & \text{ on }[0,l_{1}]\\
0 & \text{elsewhere}
\end{array}\right.
\]
and, given a cut off function $\chi:=\chi_{l}(x)$, with $l<l_1$, we define 
\begin{equation}
W_{\lambda}(x)=\chi(x)U_{\lambda}(x)\label{eq:Wlambda}
\end{equation}
and we search a solution of (\ref{eq:P}) of the form
$u=W_{\lambda}(x)+\phi$, $\phi$ being a small error in $H^{1}(\mathcal{G})$.
To improve the readability of the paper, herafter we denote 
\[
f(s):=(s^+)^{p},
\]
 so a solution of (\ref{eq:P}) can be written as 
\begin{equation}
W_{\lambda}+\phi=i_{\lambda}^{*}(f(W_{\lambda}+\phi)).\label{eq:puntofisso}
\end{equation}
We define a linear operator 
\begin{align*}
\mathcal{L}_{\lambda} & :H^{1}(\mathcal{G})\rightarrow H^{1}(\mathcal{G})\\
\mathcal{L}_{\lambda}(\phi) & =\phi-i_{\lambda}^{*}(f'(W_{\lambda})\phi)
\end{align*}
 and we recast equation (\ref{eq:puntofisso}) as 
\[
\mathcal{L}_{\lambda}(\phi)=N_{\lambda}(\phi)+R_{\lambda}
\]
where
\[
N_{\lambda}(\phi):=i_{\lambda}^{*}\left[f(W_{\lambda}+\phi)-f(W_{\lambda})-f'(W_{\lambda})\phi\right]
\]
\[
R_{\lambda}:=i_{\lambda}^{*}(f(W_{\lambda}))-W_{\lambda}.
\]
The following result implies the invertibility of $\mathcal{L}_{\lambda}$
for $\lambda$ sufficiently large.
\begin{lem}
\label{lem:Linv}There exists $\lambda_{0},c>0$ such that $\forall\lambda>\lambda_{0}$,
$\forall\phi\in H^{1}(\mathcal{G})$ it holds 
\[
\|\mathcal{L}_{\lambda}(\phi)\|_{\lambda}\ge c\|\phi\|_{\lambda}
\]
\end{lem}
\begin{proof}
We proceed by contradiction, assuming that there exist a sequence
$\lambda_{n}\rightarrow\infty$ and a sequence of functions $\phi_{n}\in H^{1}(\mathcal{G})$
such that $\|\phi_{n}\|_{\lambda}=1$ and 
\[
\|\mathcal{L}_{\lambda_{n}}(\phi_{n})\|_{\lambda_{n}}\rightarrow0.
\]
By definition of $\mathcal{L}_{\lambda}$ we have 
\[
\phi_{n}-\mathcal{L}_{\lambda_{n}}(\phi_{n})=i_{\lambda_{n}}^{*}(f'(W_{\lambda_{n}})\phi_{n})
\]
 that is 
\[
\left\{ \begin{array}{cc}
-\left(\phi_{n}-\mathcal{L}_{\lambda_{n}}(\phi_{n})\right)''+\lambda_{n}\left(\phi_{n}-\mathcal{L}_{\lambda_{n}}(\phi_{n})\right)=f'(W_{\lambda_{n}})\phi_{n} & \text{ on }\mathcal{G}\\
\sum_{e\prec \vv}\frac{d\left(\phi_{n}-\mathcal{L}_{\lambda_{n}}(\phi_{n})\right)_{e}}{dx}(\vv)=0 & \forall \vv\in V
\end{array}\right.
\]
and, set $z_{n}:=\phi_{n}-\mathcal{L}_{\lambda_{n}}(\phi_{n}),$ and
$h_{n}:=\mathcal{L}_{\lambda_{n}}(\phi_{n})$ we get 
\begin{equation}
\left\{ \begin{array}{cc}
-z_{n}''+\lambda_{n}z_{n}=f'(W_{\lambda_{n}})z_{n}+f'(W_{\lambda_{n}})h_{n} & \text{ on }\mathcal{G}\\
\sum_{e\prec \vv}\frac{dz_{e}}{dx}(\vv)=0 & \forall \vv\in V
\end{array}\right..\label{eq:zn}
\end{equation}
Also, we have 
\begin{equation}
\|z_{n}\|_{\lambda_{n}}^{2}=\|\phi_{n}\|_{\lambda_{n}}^{2}+\|\mathcal{L}_{\lambda_{n}}(\phi_{n})\|_{\lambda_{n}}^{2}-2\left\langle \phi_{n},\mathcal{L}_{\lambda_{n}}(\phi_{n})\right\rangle _{\lambda_{n}}\rightarrow1\label{eq:normazn}
\end{equation}
and, on the other hand,
\begin{align*}
\|z_{n}\|_{\lambda_{n}}^{2} & =\int_{\mathcal{G}}\left(z_{n}'\right)^{2}dx+\lambda_{n}\int_{\mathcal{G}}\left(z_{n}\right)^{2}dx\\
 & =\int_{\mathcal{G}}\left(-z_{n}''+\lambda_{n}z_{n}\right)z_{n}dx+\sum_{v\in V}\sum_{e\prec \vv}z_{n}'(\vv)z_{n}(\vv).
\end{align*}
In light of (\ref{eq:zn}) we have that $\sum_{e\prec\vv}z_{n}'(\vv)z_{n}(\vv)=0$
for all $\vv\in V$, and, since $W_{\lambda_{n}}=0$ outside the first
edge $I_{1}$, also that $-z_{n}''+\lambda_{n}z_{n}=0$ on $I_{e},$
$e\neq1$. Thus
\[
\|z_{n}\|_{\lambda_{n}}^{2}=\int_{I_{1}}\left(-z_{n}''+\lambda_{n}z_{n}\right)z_{n}dx=\int_{I_{1}}f'(W_{\lambda_{n}})z_{n}^{2}+f'(W_{\lambda_{n}})\mathcal{L}_{\lambda_{n}}(\phi_{n})z_{n}dx,
\]
and, since $\mathcal{L}_{\lambda_{n}}(\phi_{n})\rightarrow0$ in $H^{1}(\mathcal{G})$
and by (\ref{eq:normazn}), we have 
\begin{equation}
\int_{I_{1}}f'(W_{\lambda_{n}})z_{n}^{2}\rightarrow1\text{ while }n\rightarrow\infty.\label{eq:fprimozn}
\end{equation}
On the edge $I_{1}$ we consider the rescaling $s=x\sqrt{\lambda_{n}}$
and we set 
\[
\tilde{z}_{n}(s)=\lambda_{n}^{1/4}z_{n}\left(\frac{s}{\sqrt{\lambda_{n}}}\right)\text{ for }s\in[0,l_{1}\sqrt{\lambda_{n}}].
\]
Of course 
\[
\tilde{z}_{n}'(s)=\lambda_{n}^{-\frac{1}{4}}z_{n}'\left(\frac{s}{\sqrt{\lambda_{n}}}\right)\text{ and }\tilde{z}_{n}''(s)=\lambda_{n}^{-\frac{3}{4}}z_{n}''\left(\frac{s}{\sqrt{\lambda_{n}}}\right)
\]
and, recalling the definition (\ref{eq:Wlambda}) of $W_{\lambda}$,
and (\ref{eq:zn}), 
\[
-\tilde{z}_{n}''(s)+\tilde{z}_{n}'(s)=p\chi^{p-1}\left(\frac{s}{\sqrt{\lambda_{n}}}\right)U^{p-1}(s)\left[\tilde{z}_{n}(s)+\tilde{h}_{n}(s)\right]\text{ for }s\in[0,l_{1}\sqrt{\lambda_{n}}]
\]
where $\tilde{h}_{n}(s):=\lambda_{n}^{\frac{1}{4}}h_{n}\left(\frac{s}{\sqrt{\lambda_{n}}}\right)$.
Moreover it holds, for some constant $C>0$,
\begin{equation}
\|\tilde{z}_{n}\|_{H^{1}([0,l_{1}\sqrt{\lambda_{n}}])}\le C,\label{eq:zntildenorma}
\end{equation}
in fact 
\[
\int_{0}^{l_{1}\sqrt{\lambda_{n}}}\left(\tilde{z}_{n}^{'}\right)^{2}(s)+\tilde{z}_{n}^{2}(s)ds=\int_{0}^{l_{1}}\left(z_{n}^{'}\right)^{2}(x)+\lambda_{n}z_{n}^{2}(x)dx\le\|z_{n}\|_{\lambda_{n}}^{2}
\]
which is bounded by (\ref{eq:normazn}). Analogously 
\[
\|\tilde{h}_{n}\|_{H^{1}([0,l_{1}\sqrt{\lambda_{n}}])}\le\|h_{n}\|_{\lambda_{n}}\rightarrow0.
\]
By (\ref{eq:zntildenorma}) we have that there exists a function $\tilde{z}$
defined on $\mathbb{R}^{+}$ such that, fixed any $T>0$,
\begin{align*}
\tilde{z}_{n}\rightarrow\tilde{z} & \text{ a.e. in }\mathbb{R}^{+}\\
\tilde{z}_{n}\rightarrow\tilde{z} & \text{ in }L^{p}([0,T])\text{ for all }q>1\\
\tilde{z}_{n}\rightharpoonup\tilde{z} & \text{ weakly in }H^{1}([0,T]).
\end{align*}
 We can show, indeed, that $\tilde{z}\in H^{1}([0,T])$. Consider
\[
\zeta_{n}=\tilde{z}_{n}\chi\left(\frac{s}{\sqrt{\lambda_{n}}}\right).
\]
Since $\lambda_{n}\rightarrow\infty$ we have that $\|\zeta_{n}\|_{H^{1}(\mathbb{R}^{+})}\le C\|\tilde{z}_{n}\|_{H^{1}([0,l_{1}\sqrt{\lambda_{n}}])}\le C,$
thus $\zeta_{n}$ admits a weak limit in $H^{1}(\mathbb{R}^{+})$.
Also, $\zeta_{n}=\tilde{z}_{n}$ on $[0,\delta\sqrt{\lambda_{n}}]$,
so $\zeta_{n}\rightharpoonup\tilde{z}$ weakly in $H^{1}(\mathbb{R}^{+})$
and $\tilde{z}\in H^{1}(\mathbb{R}^{+})$. 

Now, take a function $\varphi\in C^{\infty}(\mathbb{R}^{+})$, and
take $T>0$ such that the support of $\varphi$ is included in $[0,T]$,
so
\begin{align*}
\int_{[0,T]}\left(-\tilde{z}_{n}''(s)\right.&\left.+\tilde{z}_{n}'(s)\right)\varphi(s)ds\\ &=\int_{[0,T]}p\left(\chi^{p-1}\left(\frac{s}{\sqrt{\lambda_{n}}}\right)U^{p-1}(s)\left[\tilde{z}_{n}(s)+\tilde{h}_{n}(s)\right]\right)\varphi(s)ds\\
 & =\int_{[0,T]}pU^{p-1}(s)\tilde{z}_{n}(s)\varphi(s)ds+o(1)
\end{align*}
Integrating by parts the first term and passing to the limit we have
that 
\begin{align*}
\int_{\mathbb{R}^{+}}\tilde{z}'(s)\varphi(s)+\tilde{z}'(s)\varphi(s)ds & =\int_{\mathbb{R}^{+}}pU^{p-1}(s)\tilde{z}(s)\varphi(s)ds.
\end{align*}
 Since $\varphi$ is arbitrary, we have that $\tilde{z}$ is a solution
of (\ref{eq:linearizzato}), so $\tilde{z}\equiv0$. Moreover, extending
by zero $\tilde{z}_{n}$ to the whole half line, we have $\tilde{z}_{n}\rightharpoonup0\text{ in }L^{2}(\mathbb{R}^{+})$,
thus 
\[
p\int_{0}^{l_{1}\sqrt{\lambda_{n}}}U^{p-1}(s)\tilde{z}_{n}^{2}(s)ds=p\int_{\mathbb{R}^{+}}U^{p-1}(s)\tilde{z}_{n}^{2}(s)ds\rightarrow0.
\]
 This leads to a contradiction in light of (\ref{eq:fprimozn}), in
fact
\begin{align*}
p\int_{0}^{l_{1}\sqrt{\lambda_{n}}}U^{p-1}(s)\tilde{z}_{n}^{2}(s)ds & \ge p\int_{0}^{l_{1}\sqrt{\lambda_{n}}}\chi^{p-1}\left(\frac{s}{\sqrt{\lambda_{n}}}\right)U^{p-1}(s)\tilde{z}_{n}^{2}(s)ds\\
 & =\int_{0}^{l_{1}}f'(W_{\lambda_{n}})z_{n}^{2}dx\rightarrow1.
\end{align*}
This concludes the proof.
\end{proof}
\begin{prop}
\label{prop:resto}We have $\|R\|_{\lambda}\le\lambda^{-\alpha}$
for any $\alpha>0$.
\end{prop}
\begin{proof}
Take $V=i_{\lambda}^{*}(f(W_{\lambda}))$. Then we have, by direct
computation, that 
\begin{align}
-(V-W_{\lambda})''(x)+\lambda(V-W_{\lambda})(x)&=  (\chi^{p}-\chi)(x)\lambda^{\frac{p}{p-1}}U^{p}(x\sqrt{\lambda})\nonumber \\
 & -\lambda^{\frac{1}{p-1}}\chi''(x)U(x\sqrt{\lambda})-2\lambda^{\frac{1}{p-1}}\sqrt{\lambda}\chi'(x)U'(x\sqrt{\lambda})\label{eq:VmenoW}
\end{align}
and $V'(0)=0.$ Thus, multiplying (\ref{eq:VmenoW}) by $V-W_{\lambda}$,
and integrating by parts we have
\begin{align*}
\|R\|_{\lambda} & =\|V-W_{\lambda}\|_{\lambda}\le C\lambda^{\frac{p}{p-1}}|(\chi^{p}-\chi)(x)U^{p}(x\sqrt{\lambda})|_{L^{2}([0,l_{1}])}\\
 & +C\lambda^{\frac{1}{p-1}}|\chi''(x)U(x\sqrt{\lambda})|_{L^{2}([0,l_{1}])}+C\lambda^{\frac{1}{p-1}}\sqrt{\lambda}|\chi'(x)U'(x\sqrt{\lambda})|_{L^{2}([0,l_{1}])}\\
 & =:I_{1}+I_{2}+I_{3}.
\end{align*}
 By a change of variables, and since $U(x)$ decays exponentially
in $x$, we have 
\begin{align*}
I_{1}^{2} & \le C\lambda^{\frac{2p}{p-1}}\int_{\delta}^{2\delta}U^{2p}(x\sqrt{\lambda})dx=C\lambda^{\frac{2p}{p-1}-\frac{1}{2}}\int_{\delta\sqrt{\lambda}}^{2\delta\sqrt{\lambda}}U^{2p}(s)ds\\
&\le C\lambda^{\frac{3p+1}{2(p-1)}}\int_{\delta\sqrt{\lambda}}^{2\delta\sqrt{\lambda}}e^{-2ps}ds \le C\lambda^{\frac{3p+1}{2(p-1)}}\left[-\frac{e^{-2ps}}{2p}\right]_{\delta\sqrt{\lambda}}^{2\delta\sqrt{\lambda}}\le C\lambda^{\frac{3p+1}{2(p-1)}}e^{-2p\delta\sqrt{\lambda}}.
\end{align*}
In the same way we can proceed for $I_{2}$ and $I_{3}$, obtaining
the claim.
\end{proof}
\begin{proof}[Proof of Theorem \ref{thm:1picco}]
We look for a solution of (\ref{eq:puntofisso}) in the form $W_{\lambda}+\phi$,
where $W_{\lambda}$ is defined in (\ref{eq:Wlambda}). This corresponds
to find a fixed point of the map 
\begin{align*}
T_{\lambda}: & H^{1}(\mathcal{G})\rightarrow H^{1}(\mathcal{G})\\
T_{\lambda}(\phi):= & \mathcal{L}_{\lambda}^{-1}\left(N_{\lambda}(\phi)+R_{\lambda}\right).
\end{align*}
We prove that $T$ is a contraction on $\left\{ \phi\in H^{1}(\mathcal{G}),\ \|\phi\|_{\lambda}\le c\lambda^{-\alpha}\right\} $
for some positive $\alpha,c$. By Lemma \ref{lem:Linv}, there exists
$c>0$ such that
\begin{eqnarray*}
\|T_{\lambda}(\phi)\|_{\lambda} & \le & c\left(\|N_{\lambda}(\phi)\|_{\lambda}+\|R_{\lambda}\|_{\lambda}\right)\\
\|T_{\lambda}(\phi_{1})-T_{\lambda}(\phi_{2})\|_{\lambda} & \le & c\left(\|N_{\lambda}(\phi_{1})-N_{\lambda}(\phi_{2})\|_{\lambda}\right).
\end{eqnarray*}
By the mean value theorem and by the properties of $i_{\lambda}^{*}$
there exists $0<\theta(x)<1$ such that
\begin{align*}
\|N_{\lambda}(\phi_{1})-N_{\lambda}(\phi_{2})\|_{\lambda}^{2} & \le c\int_{\mathcal{G}}\left[(W_{\lambda}+\phi_{2}+\theta(\phi_{1}-\phi_{2}))^{p-1}-(W_{\lambda})^{p-1}\right]^{2}\left(\phi_{1}-\phi_{2}\right)^{2}dx,
\end{align*}
 so, if $\|\phi_{i}\|_{\lambda}$ is small enough, then also $|\phi_{i}|_{L^{2}(\mathcal{G})}$
is small and we can find a constant $0<K<1$ such that 
\[
\|N_{\lambda}(\phi_{1})-N_{\lambda}(\phi_{2})\|_{\lambda}\le K\|\phi_{1}-\phi_{2}\|_{\lambda}\,.
\]
In a similar way we can prove that, if $\|\phi\|_{\lambda}$ is small
enough, by Proposition \ref{prop:resto}

\[
\|T_{\lambda}(\phi)\|_{\lambda}\le c\left(\|N_{\lambda}(\phi)\|_{\lambda}+\|R_{\lambda}\|_{\lambda}\right)\le c\left(\|\phi\|_{\lambda}^{2}+\lambda^{-\alpha}\right).
\]
Then there exists $c>0$ such that $T_{\lambda}$ maps a ball of center
$0$ and radius $c\lambda^{-\alpha}$ in $H^{1}(\mathcal{G})$ into
itself and it is a contraction. So there exists a fixed point $\phi_{\lambda}$
with norm $\|\phi_{\lambda}\|_{\lambda}=O(\lambda^{-\alpha})$. 

At this point we proved that (\ref{eq:P}) has a one-peaked solution
$u=W_{\lambda}+\phi$, with $\|\phi_{\lambda}\|_{\lambda}=O(\lambda^{-\alpha})$.
To conclude the proof we compute the $L^{2}$ norm of the solution,
that is 
\begin{align*}
|u|_{L^{2}(\mathcal{G})}^2 & =C|W_{\lambda}|_{L^{2}(\mathcal{G})}^2+l.o.t.=C\int_{0}^{l_{1}}U_{\lambda}^{2}(x)\chi^{2}(x)dx+l.o.t.\\
 & =C\lambda^{\frac{5-p}{2(p-1)}}|U|_{L^{2}(\mathbb{R}^{+})}^{2}+l.o.t.
\end{align*}
which concludes the proof. 
\end{proof}

\subsection{Multipeaked solutions}

We consider now a graph $\mathcal{G}$ which has at least $k$ vertices
$\vv_{1}\dots,\vv_{k}$ of degree $1$, and we construct a solution
of (\ref{eq:P}) which has a positive peak on any vertex $\vv_{i}$,
$i=1,\dots,k$. Without loss of generality we suppose that each vertex
$\vv_{i}$, $i=1,\dots,k$ lies on of the edge $I_{i}=[0,l_{i}]$ and
that $\vv_{i}$ corresponds to the coordinate $x=0$. 

The strategy of the proof is similar to the previous one, so
we only underline the differences. We define 
\begin{equation}
W_{\lambda}(x)=\sum_{i=1}^{k}\chi_{i}(x)U_{\lambda,i}(x)\label{eq:Wlambdai}
\end{equation}
where 
\[
U_{\lambda,i}(x)=\left\{ \begin{array}{cc}
\lambda^{\frac{1}{p-1}}U(x\sqrt{\lambda}) & \text{ on }[0,l_{i}]\\
0 & \text{elsewhere}
\end{array}\right.
\]
and, $\chi_{i}:=\chi_{\delta,i}(x)$ is a cut off function which is
$1$ on $[0,\delta/2]\subset[0,l_{i}]$ and $0$ on $[\delta,l_{i}]$
and on every other edge $I_{j}$, $j\neq i$. Here $\delta<\min_i l_i$.

It is clear that $W_{\lambda}(x)\in H^{1}(\mathcal{G})$.
As before, we search a solution of (\ref{eq:P}) of the form $u=W_{\lambda}(x)+\phi$,
$\phi$ being a small error in $H^{1}(\mathcal{G})$. We can prove
the invertibility of the operator $\mathcal{L}_{\lambda}$ as following.
\begin{lem}
\label{lem:Linv-1}There exist $\lambda_{0},c>0$ such that $\forall\lambda>\lambda_{0}$,
$\forall\phi\in H^{1}(\mathcal{G})$ it holds 
\[
\|\mathcal{L}_{\lambda}(\phi)\|_{\lambda}\ge c\|\phi\|_{\lambda}
\]
\end{lem}
\begin{proof}
As before, we proceed by contradiction, assuming that there exist
a sequence $\lambda_{n}\rightarrow\infty$ and a sequence of functions
$\phi_{n}\in H^{1}(\mathcal{G})$ such that $\|\phi_{n}\|_{\lambda}=1$
and $\|\mathcal{L}_{\lambda_{n}}(\phi_{n})\|_{\lambda_{n}}\rightarrow0.$

Setting $z_{n}:=\phi_{n}-\mathcal{L}_{\lambda_{n}}(\phi_{n})$ and $h_{n}:=\mathcal{L}_{\lambda_{n}}(\phi_{n})$,
we can prove as in Lemma \ref{lem:Linv} that $z_{n}$ solves equation
(\ref{eq:zn}) and that $\|z_{n}\|_{\lambda_{n}}^{2}\rightarrow1$
as $n\rightarrow\infty$. Since $W_{\lambda_{n}}=0$ outside the first
$k$ edges $I_{1},\dots,I_{k}$, we have 
\begin{equation}
\|z_{n}\|_{\lambda_{n}}^{2}=\sum_{i=1}^{k}\int_{I_{i}}\left(-z_{n}''+\lambda_{n}z_{n}\right)z_{n}dx=\sum_{i=1}^{k}\int_{I_{i}}f'(W_{\lambda_{n}})z_{n}^{2}dx+o(1).\label{eq:normasomma}
\end{equation}
This means that there is at least one edge $I_{\bar{i}}$ such that 
\begin{equation}
\int_{I_{\bar{i}}}f'(W_{\lambda_{n}})z_{n}^{2}dx\not\rightarrow0\label{eq:nonzero}\,.
\end{equation}
Letting now $z_{n,\bar{i}}=\left.z_{n}\right|_{I_{\bar{i}}}$, we can
define the functions
\[
\tilde{z}_{n}(s)=\lambda_{n}^{1/4}z_{n,\bar{i}}\left(\frac{s}{\sqrt{\lambda_{n}}}\right)\text{ for }s\in[0,l_{\bar{i}}\sqrt{\lambda_{n}}]
\]
and we can repeat the argument of Lemma \ref{lem:Linv} to prove that
$\tilde{z}_{n}\rightharpoonup0\text{ in }L^{2}(\mathbb{R}^{+})$ as
$n\rightarrow\infty$. This contradicts (\ref{eq:nonzero}).
\end{proof}
\begin{prop}
\label{prop:resto-1}We have $\|R\|_{\lambda}\le\lambda^{-\alpha}$
for any $\alpha>0$.
\end{prop}
\begin{proof}
As in Proposition \ref{prop:resto}, we take $V=i_{\lambda}^{*}(f(W_{\lambda}))$,
where $W_{\lambda}$ is defined in (\ref{eq:Wlambdai}). Then we find
that $V-W_{\lambda}$ solves the following differential equation
\begin{multline}
-(V-W_{\lambda})''(x)+\lambda(V-W_{\lambda})(x)=\sum_{i=1}^{k}(\chi_{i}^{p}-\chi_{i})(x)\lambda^{\frac{p}{p-1}}U^{p}(x\sqrt{\lambda})\\
-\lambda^{\frac{1}{p-1}}\sum_{i=1}^{k}\chi_{i}''(x)U(x\sqrt{\lambda})-2\lambda^{\frac{1}{p-1}}\sqrt{\lambda}\sum_{i=1}^{k}\chi_{i}'(x)U'(x\sqrt{\lambda})\label{eq:VmenoW-1}
\end{multline}
which leads to the same conclusion of Proposition \ref{prop:resto}.
\end{proof}
\begin{proof}[Proof of Theorem \ref{thm:multi}]
The proof of this theorem is verbatim the proof of Theorem \ref{thm:1picco}. 
\end{proof}


\begin{thebibliography}{99}
	
	\bibitem{ACFN-jde14} Adami R., Cacciapuoti C., Finco D., Noja D., \textit{Variational properties and orbital stability of standing waves for NLS equation on a star graph}, J. Differential Equations \textbf{257}, no. 10 (2014),  3738--3777.
	
	\bibitem{ACFN-aihp14} Adami R., Cacciapuoti C., Finco D., Noja D., \textit{Constrained energy minimization and orbital stability for the NLS equation on a star graph},  Ann. I. H.  Poincar\'e --  AN \textbf{31} (2014), 1289--1310.
	
	\bibitem{AD}
	Adami R., Dovetta S., {\em One-dimensional versions of three-dimensional system: ground states for the NLS on the spatial grid}, Rendiconti di Matematica e delle sue applicazioni, {\bfseries 39} 7 (2018), 181--194.
	
	\bibitem{ADST}
	Adami R., Dovetta S., Serra E., Tilli P., {\em Dimensional crossover with a continuum of critical exponents for NLS on doubly periodic metric graphs}, Anal. PDE, Vol. {\bf12} (2019), No. 6, 1597--1612.
	
	\bibitem{AST-CVPDE}
	Adami R., Serra E., Tilli P.,
	{\em NLS ground states on graphs},
	Calc. Var. PDE {\bf 54} (2015), no. 1, 743--761.
	
	\bibitem{AST-JFA}
	Adami R., Serra E., Tilli P.,
	{\em Threshold phenomena and existence results for NLS ground states on metric graphs},
	J. Funct. Anal. {\bf 271} (2016), no. 1, 201--223.  
	
	\bibitem{AST-CMP}
	Adami R., Serra E., Tilli P., 
	{\em Negative energy ground states for the $L^2$--critical NLSE on metric graphs},
	Comm. Math. Phys. {\bf 352} (2017), no. 1, 387--406.
	
	\bibitem{AST-bound} Adami R., Serra E., Tilli P., {\em Multiple positive bound states for the subcritical NLS equation on metric graphs}, Calc. Var. (2019) 58:5. https://doi.org/10.1007/s00526-018-1461-4.
	
	\bibitem{BK}
	Berkolaiko G., Kuchment P.,
	\emph{Introduction to quantum graphs},
	Mathematical Surveys and Monographs 186, American Mathematical Society, Providence, RI, 2013.
	
	\bibitem{BCT}
	Borrelli W., Carlone R., Tentarelli L.,
	{\em Nonlinear Dirac equation on graphs with localized nonlinearities: bound states and nonrelativistic limit},
	SIAM J. Math. An., {\bf51} 2 (2019), 1046--1081.
	
	\bibitem{BCT1}
	Borrelli W., Carlone R., Tentarelli L.,
	{\em An overview on the standing waves of nonlinear Schr�dinger and Dirac equations on metric graphs with localized nonlinearity}, Symmetry {\bf11} (2019), no. 2, art. num. 169, 22p.
	
	\bibitem{CF}
	Cacciapuoti C., Finco, D.,{\em Graph-like models for thin waveguides with Robin boundary conditions}, Asymptot. Anal. {\bf70} (2010), 199--230.
	
	\bibitem{CDS}
	Cacciapuoti C., Dovetta S., Serra E.,{\em Variational and stability properties of constant solutions to the NLS equation on compact metric graphs}, Milan Journal of Mathematics, \textbf{86}(2) (2018), 305--327.
	
	\bibitem{D-jde}
	Dovetta S., {\em Existence of infinitely many stationary solutions of the $L^2$--subcritical and critical NLSE on compact metric graphs}, J. Differential Equations \textbf{264} (2018), no. 7, 4806--4821.
	
	\bibitem{D-nodea}
	Dovetta S., {\em Mass--constrained ground states of the stationary NLSE on periodic metric graphs}, Non. Diff. Eq. Appl., to appear. ArXiv:1811.06798 [math.AP] (2018).
	
	\bibitem{DT-p}
	Dovetta S., Tentarelli L.,
	{\em Ground states of the $L^2$-critical NLS equation with localized nonlinearity on a tadpole graph}, Operator Theory: Advances and Applications, to appear. 
	ArXiv:1804.11107 [math.AP] (2018).
	
	\bibitem{DT}
	Dovetta S., Tentarelli L., {\em $L^2$--critical NLS on noncompact metric graphs with localized nonlinearity: topological and metric features}, Calc. Var. PDE {\bf58} 3 (2019) 58:108, https://doi.org/10.1007/s00526-019-1565-5.
	
	\bibitem{EP}
	Exner P., Post O., {\em Approximation of quantum graph vertex couplings by scaled Schr\"odinger operators
	on thin branched manifolds}, J. Phys. A: Math. Theo. {\bf 42} (2009), 415305, 22pp.
	
	\bibitem{Fife}
	Fife P.C., {\em Semilinear elliptic boundary value problems	with small parameters}, 
	Arch. Ration. Mech. Anal. {\bf 52} 3 (1973), 205--232.
	
	\bibitem{grieser}
	Grieser D., {\em Spectra of graph neighborhoods and scattering}, Proc. Lon. Math. Soc. {\bf97}, no. 3 (2008), 718--752.
	
	\bibitem{GKP}
	Goodman R.H., Kairzhan A., Pelinovsky D.E.,
	{\em Drift of spectrally stable shifted states on star graphs},
	arXiv:1902.03612 [math.AP] (2019).
	
	\bibitem{KP-JDE}
	Kairzhan A., Pelinovsky D.E.,
	{\em Nonlinear instability of half-solitons on star graphs},
	J. Differential Equations {\bf 264} (2018), no. 12, 7357--7383.
	
	\bibitem{KP-JPA}
	Kairzhan A., Pelinovsky D.E.,
	{\em Spectral stability of shifted states on star graphs}, 
	J. Phys. A: Math. Theor. {\bf51} (2018)
	095203.
	
	\bibitem{Kw}
	 Kwong, M.K.,
	{\em Uniqueness of positive solutions of  $\Delta u-u+u^p = 0$ in $\R^n$}, 
	Arch. Ration. Mech. Anal. {\bf 105} (1989), 243--266
	
	\bibitem{MP-AMRX}
	Marzuola J.L., Pelinovsky D.E.,
	{\em Ground state on the dumbbell graph},
	Applied Mathematics Research eXpress, {\bf2016} 1 (2016), 98--145.
	
	\bibitem{molchanov}
	Molchanov S., Vainberg B., {\em Scattering Solutions in Networks of Thin Fibers: Small Diameter Asymptotics}, Commun. Math. Phys. {\bf 273}, no. 2 (2007), 533--559.
	
	\bibitem{MNS}
	Mugnolo D., Noja D., Seifert C., {\em Airy-type evolution equations on star graphs},
	Anal. PDE {\bf11} (2018), no. 7, 1625--1652.
	
	\bibitem{N}
	Noja D., {\em Nonlinear Schr\"odinger equation on graphs: recent results and open problems}, Phil. Trans. R.
	Soc. A, {\bf372} (2014), 20130002 (20 pages).
	
	\bibitem{NPS}
	Noja D.,  Pelinovsky D.E., Shaikhova G., {\em Bifurcations and stability of standing waves in the nonlinear Schr\"odinger equation on the tadpole graph}, 
	Nonlinearity {\bf28} (2015),  2343--2378.
	
	\bibitem{Pa} Pankov A.,
	\textit{Nonlinear Schr\"odinger equations on periodic metric graphs}, 
	Discrete Contin. Dyn. Syst. {\bf 38} (2018), no. 2, 697--714.
	
	\bibitem{PS} Pelinovsky D.E., Schneider G., 
	\textit{Bifurcations of Standing Localized Waves on Periodic Graphs}, 
	Ann. H. Poincar\'e {\bf 18 (4)} (2017), 1185--1211.
	
	\bibitem{RS} Ruedenberg K.,  Scherr  C. W., \textit{Free-Electron Network Model for Conjugated Systems. I. Theory}, J. Chem.  Phys. \textbf{21}, no. 9 (1953), 1565--1581.
	
	\bibitem{ST-JDE}
	Serra E., Tentarelli L., 
	{\em Bound states of the NLS equation on metric graphs with localized nonlinearities},
	J. Differential Equations {\bf 260} (2016), no. 7, 5627--5644.
	
	\bibitem{ST-NA}
	Serra E., Tentarelli L.,
	{\em On the lack of bound states for certain NLS equations on metric graphs},
	Nonlinear Anal. {\bf 145} (2016), 68--82.
	
	\bibitem{T-JMAA}
	Tentarelli L., {\em NLS ground states on metric graphs with localized nonlinearities},
	J. Math. Anal. Appl. {\bf 433} (2016), no. 1, 291--304.
	
\end{thebibliography}
\end{document}